\documentclass{amsart}
\usepackage[english]{babel}
\usepackage{amsfonts, amssymb, amsmath}
\usepackage{xcolor}
\usepackage{comment}


\title{Periodic Shadowing for Set-Valued Maps}

\author[M. Oliveira]{M.C. Oliveira}
\address{Universidade Estadual do Maranhão, 65055-310, São Luís, Brazil.}
\email{marlonoliveira@professor.uema.br}

\newtheorem{theorem}{Theorem}[section]
\newtheorem{corollary}[theorem]{Corollary}
\newtheorem{lemma}[theorem]{Lemma}
\newtheorem{proposition}[theorem]{Proposition}

\theoremstyle{definition}
\newtheorem{definition}[theorem]{Definition}
\newtheorem{remark}[theorem]{Remark}
\newtheorem{ex}[theorem]{Example}

\begin{document}

\begin{abstract}
We study shadowing-type properties for set-valued dynamical systems. In particular, we investigate the periodic shadowing property and its relationship with expansivity and chain transitivity. We establish that for positively expansive set-valued maps on compact metric spaces, the shadowing property implies the periodic shadowing property. Furthermore, we show that for upper semicontinuous chain transitive set-valued maps, periodic shadowing implies both shadowing and topological transitivity. We also present a general construction of set-valued maps with shadowing arising from single-valued systems admitting an isometric involution. Several examples, including systems from symbolic dynamics, are provided to illustrate the theory.
\end{abstract}

\maketitle

\section{Introduction}

A discrete dynamical system is defined as a pair $(X,f)$, where $X$ is a compact metric space and $f:X\to X$ is a continuous map. A natural generalization consists of replacing $f$ by a set-valued map $F:X\to 2^X$, where $2^X$ denotes the family of all nonempty closed subsets of $X$. Such systems have attracted considerable attention due to their applications to differential inclusions, control systems, symbolic dynamics, and uncertainty models; see, for example, \cite{Guzik2017,RT,CordeiroPacifico2016,HuangLiang2017}.

The periodic shadowing property strengthens the classical shadowing property by requiring periodic pseudo-orbits to be shadowed by genuine periodic orbits. In the single-valued setting, it is closely related to expansivity, recurrence, and transitivity-type properties, playing an important role in the global organization of dynamical systems \cite{Koscielniak1996,Koscielniak2005,DarabiForouzanfar2017,OsipovPilyuginTikhomirov2010,Lee2013}.

Shadowing for set-valued maps has attracted increasing attention in recent years. In particular, Raines and Tennant \cite{RT} and Zhengyu \cite{Zhengyu} investigated shadowing and inverse limit structures for set-valued systems, establishing similarities and important differences with the classical theory. However, the periodic shadowing property remains poorly understood in the set-valued setting, especially regarding its relationship with expansivity, chain transitivity, and topological transitivity.

The purpose of this paper is to investigate the periodic shadowing property for set-valued dynamical systems and clarify its relationship with shadowing and transitivity-type properties. Due to the nonuniqueness of trajectories, extending orbit tracing properties to the set-valued setting is far from straightforward. Our main contributions are threefold. First, we prove that for positively expansive set-valued maps on compact metric spaces, shadowing implies periodic shadowing. Second, we show that for chain transitive set-valued maps, periodic shadowing implies both shadowing and topological transitivity, extending fundamental implications from the classical theory. We also investigate permanence properties of periodic shadowing, including invariance under iteration, inversion, and Cartesian products and provide examples illustrating the necessity of the imposed assumptions.

The remainder of the paper is organized as follows. In Section 2, we present preliminaries regarding topological dynamics for both single-valued and set-valued maps. Section 3 is dedicated to introduce the construction of set-valued maps via isometric involutions and discuss the preservation of fundamental dynamical properties. In Section 4, we establish our  contributions, providing the proofs for Theorem~\ref{theoremA}  and Theorem~\ref{theoremB}. Section 5 explores the permanence and stability properties of the periodic shadowing property. Finally, in Section 6, we present examples to illustrate the scope and applicability of our results.

\section{Preliminaries}
\subsection{Classical notions}

In this subsection, we recall some fundamental notions and results from topological dynamics that will serve as motivation for our set-valued extensions.

A sequence $\{x_i\}_{i\in\mathbb{Z}}$ is called an \emph{orbit} of $(X,f)$ if
\[
x_{i+1}=f(x_i)
\quad \text{for all } i\in\mathbb{Z}.
\]
A sequence $\{x_i\}_{i\in\mathbb{Z}}$ is a \emph{$\delta$-pseudo-orbit} if
\[
d(f(x_i),x_{i+1})<\delta
\quad \text{for all } i\in\mathbb{Z}.
\]
Similarly, a finite sequence $\{x_i\}_{i=0}^{n}$ is called a \emph{$\delta$-chain} of length $n$ if
\[
d(f(x_i),x_{i+1})<\delta
\quad \text{for all } 0\leq i\leq n-1.
\]

A dynamical system $(X,f)$ has the \emph{shadowing property} if for every $\varepsilon>0$ there exists $\delta>0$ such that every $\delta$-pseudo-orbit is $\varepsilon$-shadowed by some point $x\in X$, that is,
\[
d(f^i(x),x_i)<\varepsilon
\quad \text{for all } i\in\mathbb{Z}.
\]
For further details on shadowing, see \cite{AokiHiraide}.

An important variation of shadowing is the \emph{periodic shadowing property}, introduced by Ko\'scielniak in \cite{Koscielniak1996}. A homeomorphism $f:X\to X$ is said to have the periodic shadowing property if for every $\varepsilon>0$ there exists $\delta>0$ such that every periodic $\delta$-pseudo-orbit is $\varepsilon$-shadowed by a periodic point $p\in X$, namely,
\[
d(f^i(p),x_i)<\varepsilon
\quad \text{for all} i\in\mathbb{Z}.
\]
This property is not automatic; for example, irrational rotations fail to satisfy it.

A dynamical system $(X,f)$ is \emph{topologically transitive} if for every pair of nonempty open sets $U,V\subset X$, there exists $n\in\mathbb{Z}$ such that
\[
f^n(U)\cap V\neq\emptyset.
\]
Moreover, $(X,f)$ is \emph{chain transitive} if for every $\delta>0$ and every pair of points $x,y\in X$, there exists a $\delta$-chain joining $x$ to $y$.

We conclude this subsection with expansivity, introduced by Utz in \cite{Utz}, which plays a fundamental role in the relationship between shadowing and periodic shadowing. A homeomorphism $f:X\to X$ on a compact metric space $X$ is \emph{expansive} if there exists $\delta>0$ such that for any two distinct points $x,y\in X$, one can find $n\in\mathbb{Z}$ satisfying
\[
d(f^n(x),f^n(y))>\delta.
\]

The following result, due to Darabi and Forouzanfar, establishes a connection between shadowing and periodic shadowing under expansivity.

\begin{theorem}[\cite{DarabiForouzanfar2017}]
\label{DTheA}
Let $X$ be a compact metric space and let $f:X\to X$ be an expansive homeomorphism. If $f$ has the shadowing property, then $f$ also has the periodic shadowing property.
\end{theorem}


The converse implication requires an additional recurrence assumption.

\begin{theorem}[\cite{DarabiForouzanfar2017}]
Let $X$ be a compact metric space and let $f:X\to X$ be a chain transitive dynamical system. If $f$ has the periodic shadowing property, then $f$ has the shadowing property and is topologically transitive.
\end{theorem}

The results above motivate the questions addressed in this paper, namely whether analogous implications remain valid in the set-valued setting.

\subsection{Set-Valued Maps}

We now recall some basic notions concerning set-valued maps and set-valued dynamical systems.

Let $(X,d)$ be a metric space. For each $x\in X$ and $A\subset X$, define
\[
d(x,A)=\inf\{d(x,y)\; ;\; y\in A\}.
\]

We denote by $2^X$ the family of all nonempty compact subsets of $X$. The \emph{Hausdorff metric} on $2^X$ is defined by
\[
d_H(A,B)=
\max\left\{
\sup_{a\in A} d(a,B),
\sup_{b\in B} d(b,A)
\right\}.
\]

The following continuity notions are standard in set-valued analysis.

\begin{definition}
Let $X$ be a metric space with metric $d$. A set-valued map $F:X\to 2^X$ is called:
\begin{itemize}
    \item[(I)] \textbf{Upper semicontinuous} if for every $x\in X$ and every open set $U\supset F(x)$, there exists an open neighborhood $V$ of $x$ such that
    \[
    F(y)\subset U
    \quad \text{for all } y\in V.
    \]

    \item[(II)] \textbf{Lower semicontinuous} if for every $x\in X$ and every open set $U$ satisfying
    \[
    U\cap F(x)\neq\emptyset,
    \]
    there exists an open neighborhood $V$ of $x$ such that
    \[
    F(y)\cap U\neq\emptyset
    \quad \text{for all } y\in V.
    \]

    \item[(III)] \textbf{Continuous} if it is both upper and lower semicontinuous.
\end{itemize}
\end{definition}

A pair $(X,F)$, where $F:X\to 2^X$, is called a \emph{set-valued dynamical system}.

Since set-valued maps need not be invertible, we restrict our attention to forward orbits. An \emph{orbit} of a point $x\in X$ under $F$ is a sequence $\{x_i\}_{i\in\mathbb{N}}$ satisfying
\[
x_{i+1}\in F(x_i)
\quad \text{for all } i\in\mathbb{N}.
\]
A point $x\in X$ is called \emph{periodic} if there exist an orbit $\{x_i\}_{i\in\mathbb{N}}$ and an integer $m\in\mathbb{N}^*$ such that
\[
x_0=x
\quad \text{and} \quad
x_{n+m}=x_n
\quad \text{for all } n\in\mathbb{N}.
\]

We also define the inverse image associated with a set-valued map.

\begin{definition}
For each $y\in X$, define
\[
F^{-1}(y):=
\{x\in X \; ; \; y\in F(x)\}.
\]
The set $F^{-1}(y)$ is called the \emph{inverse image} of $y$ under $F$.
\end{definition}

We next recall dynamical notions that will be relevant throughout the paper.

\begin{definition}[\cite{RW}]
A set-valued map $F:X\to 2^X$ is called \emph{positively expansive} if there exists $\alpha>0$ such that for every pair of distinct points $x,y\in X$, and for any orbits $\{x_n\}_{n\in\mathbb{N}}$ and $\{y_n\}_{n\in\mathbb{N}}$ of $F$, there exists $n\in\mathbb{N}$ satisfying
\[
d(x_n,y_n)>\alpha.
\]
\end{definition}

To obtain converse-type implications, an additional recurrence assumption is generally required, namely chain transitivity; see \cite{WongSalleh2019,WS}.

\begin{definition}
Let $F:X\to 2^X$ be a set-valued map on a metric space $X$.

\begin{enumerate}
\item The map $F$ is said to be \emph{topologically transitive} if for every pair of nonempty open sets $U,V\subset X$, there exist an orbit $\{x_n\}_{n\in\mathbb{N}}$ and an integer $n\geq 0$ such that
\[
x_0\in U
\quad \text{and} \quad
x_n\in V.
\]

\item The map $F$ is said to be \emph{chain transitive} if for every $\delta>0$ and every pair of points $x,y\in X$, there exists a $\delta$-chain joining $x$ to $y$.
\end{enumerate}
\end{definition}

We now present shadowing and related notions for set-valued maps; see \cite{RT,MMT}.

\begin{definition}
Let $F:X\to 2^X$ be a set-valued map on a metric space $X$.

\begin{enumerate}
    \item Given $\delta>0$, a sequence $\{x_i\}_{i\in\mathbb{N}}$ is called a \emph{$\delta$-pseudo-orbit} of $F$ if
    \[
    d(x_{i+1},F(x_i))<\delta
    \quad \text{for all } i\in\mathbb{N}.
    \]

    \item The map $F$ has the \emph{shadowing property} if for every $\varepsilon>0$ there exists $\delta>0$ such that every $\delta$-pseudo-orbit $\{x_i\}_{i\in\mathbb{N}}$ is $\varepsilon$-shadowed by an orbit $\{z_i\}_{i\in\mathbb{N}}$, that is,
    \[
    d(z_i,x_i)<\varepsilon
    \quad \text{for all } i\in\mathbb{N}.
    \]

    \item The map $F$ has the \emph{finite shadowing property} if for every $\varepsilon>0$ there exists $\delta>0$ such that every finite $\delta$-pseudo-orbit
    \[
    \{x_i\}_{i=0}^{k-1}
    \]
    is $\varepsilon$-shadowed by a finite orbit
    \[
    \{z_0,z_1,\ldots,z_{k-1}\},
    \]
    satisfying
    \[
    d(z_i,x_i)<\varepsilon
    \quad \text{for all } 0\leq i\leq k-1.
    \]
\end{enumerate}
\end{definition}

We conclude this subsection by introducing the periodic shadowing property in the set-valued setting.

\begin{definition}
A set-valued map $F:X\to 2^X$ is said to have the \emph{periodic shadowing property} if for every $\varepsilon>0$ there exists $\delta>0$ such that every periodic $\delta$-pseudo-orbit $\{x_i\}_{i\in\mathbb{N}}$ is $\varepsilon$-shadowed by a periodic point $z\in X$ with a periodic orbit $\{z_i\}_{i\in\mathbb{N}}$, namely,
\[
d(z_i,x_i)<\varepsilon
\quad \text{for all } i\in\mathbb{N}.
\]
\end{definition}

\section{A Construction via Isometric Involutions}

We now introduce a general construction that produces set-valued dynamical systems from single-valued maps through isometric involutions. This construction provides a flexible mechanism for generating set-valued systems inheriting relevant dynamical properties from the underlying deterministic dynamics.

\begin{definition}
Let $(X,d)$ be a metric space. A map $R:X\to X$ is called an \emph{isometric involution} if $
R^2=\mathrm{id}$ and $
d(R(x),R(y))=d(x,y)$ for all $x,y\in X$.
\end{definition}

Given a map $f:X\to X$ and an isometric involution $R:X\to X$ satisfying $
R\circ f=f\circ R$, we define the set-valued map $F(x)=\{f(x),R(f(x))\}$.

The next result shows that several properties are preserved under this construction, in particular shadowing.

\begin{theorem}[Construction via Isometric Involutions]
\label{thm:involution}
Let $(X,d)$ be a metric space and let $f:X\to X$ be a map. Assume that $R:X\to X$ is an isometric involution satisfying $
R\circ f=f\circ R$.
Define the set-valued map $F:X\to 2^X$ by $F(x)=\{f(x),R(f(x))\}$.

Then the following assertions hold:

\begin{enumerate}
    \item If $f$ has the shadowing property, then $F$ has the shadowing property.

    \item If $f$ has the periodic shadowing property, then $F$ has the periodic shadowing property.

    \item If $f$ is topologically transitive, then $F$ is topologically transitive.

    \item If $f$ is chain transitive, then $F$ is chain transitive.
\end{enumerate}
\end{theorem}

\begin{proof}

We first prove (1). Let $\varepsilon>0$ and let $\delta>0$ be given by the shadowing property of $f$. Let $
\{x_n\}_{n\in\mathbb N}$ be a $\delta$-pseudo-orbit of $F$. We inductively construct a $\delta$-pseudo-orbit $
\{y_n\}_{n\in\mathbb N}$ of $f$.

Choose $y_0\in\{x_0,R(x_0)\}$. Assume first that $y_0=x_0$. Since $d(x_1,F(x_0))<\delta$, if $d(x_1,f(x_0))<\delta$, set $y_1=x_1$; otherwise define $
y_1=R(x_1)$. In this case,
\[
d(R(x_1),f(x_0))
=
d(x_1,R(f(x_0)))
<
\delta,
\]
since $R$ is an isometric involution commuting with $f$. Hence
$d(y_1,f(y_0))<\delta$. The case $y_0=R(x_0)$ is analogous.

Proceeding inductively, suppose that $y_n\in\{x_n,R(x_n)\}$ has been chosen so that $d(y_n,f(y_{n-1}))<\delta$. If $d(x_{n+1},f(x_n))<\delta$, define
$y_{n+1}=x_{n+1}$ when $y_n=x_n$, and $y_{n+1}=R(x_{n+1})$ otherwise. 

If instead $d(x_{n+1},R(f(x_n)))<\delta$, reverse the choice. Since $R$ is an isometric involution commuting with $f$, we obtain $d(y_{n+1},f(y_n))<\delta$ in either case. Thus, $\{y_n\}_{n\in\mathbb N}$ is a $\delta$-pseudo-orbit of $f$.

By the shadowing property of $f$, there exists $z\in X$ such that $
d(f^n(z),y_n)<\varepsilon$ for all $n\in\mathbb N$. 

Define $w_n=f^n(z)$ if $y_n=x_n$, and $w_n=R(f^n(z))$ if $y_n=R(x_n)$. Since
$R\circ f=f\circ R$, the sequence $\{w_n\}_{n\in\mathbb N}$ is an orbit of $F$, and 
\[
d(x_n,w_n)
=
d(y_n,f^n(z))
<
\varepsilon
\]
for all $n\in\mathbb N$, proving (1).

Assertion (2) follows similarly. Given a periodic $\delta$-pseudo-orbit of $F$, the above construction produces a periodic $\delta$-pseudo-orbit of $f$. By periodic shadowing, there exists a periodic orbit of $f$ shadowing it, which induces a periodic orbit of $F$.

Now, let $U,V\subset X$ be nonempty open sets. Since $f$ is topologically transitive, there exist $x\in U$ and $n\geq 0$ such that $f^n(x)\in V$.

We construct an orbit of $F$ by choosing $x_{k+1}=f(x_k)$ for every $k\geq 0$. Since $f(x_k)\in F(x_k)$, the sequence $\{x_k\}$ is an orbit of $F$.

Moreover, $x_0=x\in U$ and $x_n=f^n(x)\in V$. Hence $F$ is topologically transitive.

Consider $\delta>0$ and let $x,y\in X$. Since $f$ is chain transitive, there exists a $\delta$-chain
\[
x=x_0,x_1,\ldots,x_n=y
\]
such that $d(f(x_i),x_{i+1})<\delta$ for all $0\leq i\leq n-1$. Since $f(x_i)\in F(x_i)$, we obtain
\[
d(F(x_i),x_{i+1})
\leq d(f(x_i),x_{i+1})
<\delta.
\]
Thus,
\[
x=x_0,x_1,\ldots,x_n=y
\]
is also a $\delta$-chain for $F$. Therefore, $F$ is chain transitive.
\end{proof}

This construction requires neither compactness of the phase space nor continuity of the induced set-valued map, making it particularly suitable for generating examples in general metric settings. Moreover, it provides a systematic method for transferring dynamical properties from deterministic systems to the set-valued setting.

\section{Main Theorem : Shadowing VS. Periodic Shadowing}

We begin by establishing a set-valued analogue of a classical implication between shadowing and periodic shadowing under expansivity.

\begin{theorem}    
\label{theoremA}
Let $F:X\to 2^X$ be a positively expansive set-valued map on a compact metric space $X$. If $F$ has the shadowing property, then $F$ also has the periodic shadowing property.
\end{theorem}

\begin{proof}
Let $\alpha>0$ be an expansivity constant for $F$, and choose $\varepsilon<\frac{\alpha}{2}$. By the shadowing property, there exists $\delta>0$ such that every $\delta$-pseudo-orbit is $\varepsilon$-shadowed by an orbit of $F$.

Let $\{p_n\}_{n\in\mathbb N}$ be a periodic $\delta$-pseudo-orbit of period $k$, that is, $p_{n+k}=p_n$ for all $n\in\mathbb N$. 
By shadowing, there exists an orbit $\{q_n\}_{n\in\mathbb N}$ of $F$ such that $
d(p_n,q_n)<\varepsilon$ for all $n\in\mathbb N$.

Define $r_n=q_{n+k}$ for all $n\in\mathbb N$. Since $\{q_n\}_{n\in\mathbb N}$ is an orbit of $F$, so is $\{r_n\}_{n\in\mathbb N}$. 

Moreover, using the periodicity of $\{p_n\}_{n\in\mathbb N}$,
\[ d(r_n,q_n)
=
d(q_{n+k},q_n) \leq
d(q_{n+k},p_{n+k})
+
d(p_n,q_n)
<
2\varepsilon
<
\alpha.
\]
Since both $\{q_n\}_{n\in\mathbb N}$ and $\{r_n\}_{n\in\mathbb N}$ are orbits of $F$, positive expansivity implies that $q_0=r_0$. Hence $q_{n+k}=q_n$ for all $n\in\mathbb N$, showing that $\{q_n\}_{n\in\mathbb N}$ is periodic. Therefore, $F$ has the periodic shadowing property.
\end{proof}

The previous theorem naturally raises the converse question. We now show that, under chain transitivity, periodic shadowing implies both shadowing and topological transitivity.

\begin{theorem}
\label{theoremB}
Let $F:X\to 2^X$ be an upper semicontinuous chain transitive set-valued map. If $F$ has the periodic shadowing property, then $F$ has the shadowing property and is topologically transitive.
\end{theorem}

We first recall the following characterization of shadowing for set-valued maps.

\begin{lemma}[\cite{Zhengyu}]
\label{lemma:FSP}
Let $F:X\to 2^X$ be an upper semicontinuous set-valued map. Then $F$ has the shadowing property if and only if it has the finite shadowing property.
\end{lemma}

\begin{proof}[Proof of Theorem~\ref{theoremB}]
We first prove that $F$ has the shadowing property.

Let $\varepsilon>0$ be arbitrary. By the periodic shadowing property, there exists $\delta>0$ such that every periodic $\delta$-pseudo-orbit is $\varepsilon$-shadowed by a periodic orbit.

Consider an arbitrary finite $\delta$-pseudo-orbit $\{x_i\}_{i=0}^{n}$.
Since $F$ is chain transitive, there exists a $\delta$-chain from $x_n$ to $x_0$, say 
\[
y_0=x_n,\, y_1,\, \ldots,\, y_m=x_0,
\]
satisfying $d(F(y_i),y_{i+1})<\delta$ for all $0\leq i\leq m-1$.

Consequently,
\[
\{x_0,x_1,\ldots,x_n,y_1,\ldots,y_{m-1}\}
\]
forms a periodic $\delta$-pseudo-orbit.

By the periodic shadowing property, there exists a periodic orbit $\{z_i\}_{i\in\mathbb N}$ such that $d(z_i,\eta_i)<\varepsilon$, where $\eta$ denotes the above periodic pseudo-orbit. In particular, $d(z_i,x_i)<\varepsilon$
for all $0\leq i\leq n$.

Thus, $F$ has the finite shadowing property. By Lemma \ref{lemma:FSP}, $F$ has the shadowing property.

We now prove topological transitivity.

Let $U,V\subset X$ be nonempty open sets, and choose points $x\in U$ and $y\in V$. Let
\[
\varepsilon
<
\min\{
d(x,X\setminus U),
d(y,X\setminus V)
\}.
\]
Choose $\delta>0$ associated to $\varepsilon$ by periodic shadowing. By chain transitivity, there exists a $\delta$-chain
\[
x=x_0,x_1,\ldots,x_n=y
\]
from $x$ to $y$.

Since $F$ has shadowing property also has finite shadowing property, so  there exists a orbit $\{z_i\}$ satisfies $d(z_0,x)<\varepsilon$ and $d(z_n,y)<\varepsilon$. 

Hence $z_0\in U$ and $z_n\in V$. Since $\{z_i\}$ is an orbit of $F$, this proves that $F$ is topologically transitive.
\end{proof}

As an immediate consequence of Theorems~\ref{theoremA} and \ref{theoremB}, we obtain the following characterization.

\begin{corollary}
Let $F:X\to 2^X$ be a positively expansive and chain transitive set-valued map. Then $F$ has the shadowing property if and only if it has the periodic shadowing property.
\end{corollary}

\section{Properties of Periodic Shadowing Property}

We establishing several stability properties of the periodic shadowing property.To investigate permanence properties of periodic shadowing, we first introduce a finite version of the property.

\begin{definition}
Let $(X,d)$ be a metric space and let $F:X\to 2^X$ be a set-valued map. Given $\delta>0$ and $N\in\mathbb N$, a finite sequence $\{x_i\}_{i=0}^{N-1}$ is called a \emph{finite periodic $\delta$-pseudo-orbit} of $F$ if $
d(x_{i+1},F(x_i))<\delta$ for every $0\leq i\leq N-2$, and $d(x_0,F(x_{N-1}))<\delta$. 

$F:X\to 2^X$ is said to have the \emph{finite periodic shadowing property} if for every $\varepsilon>0$ there exists $\delta>0$ such that every finite periodic $\delta$-pseudo-orbit $\{x_i\}_{i=0}^{N-1}$ is $\varepsilon$-shadowed by a periodic orbit $\{z_i\}_{i=0}^{N-1}$ satisfying 
\[
z_{i+1}\in F(z_i),
\qquad
z_0\in F(z_{N-1}),
\]
and $d(z_i,x_i)<\varepsilon$ for all $0\leq i\leq N-1$.
\end{definition}

The next lemma shows that finite periodic shadowing is equivalent to periodic shadowing.

\begin{lemma}
\label{finitePSP}
Let $(X,d)$ be a metric space and let $F:X\to 2^X$ be a set-valued map. Then $F$ has the periodic shadowing property if and only if it has the finite periodic shadowing property.
\end{lemma}

\begin{proof}
The necessity is immediate.

Conversely, suppose that $F$ has the finite periodic shadowing property. Let $\varepsilon>0$, and choose $\delta>0$ associated to $\varepsilon$ by this property.

Let $\{x_n\}_{n\in\mathbb N}$ be a periodic $\delta$-pseudo-orbit of $F$ with period $N$, namely, $x_{n+N}=x_n$ for all $n\in\mathbb N$. Then $\{x_i\}_{i=0}^{N-1}$ is a finite periodic $\delta$-pseudo-orbit of $F$.

By assumption, there exists a periodic orbit $\{z_i\}_{i=0}^{N-1}$ such that
\[
z_{i+1}\in F(z_i),
\qquad
z_0\in F(z_{N-1}),
\]
and $d(z_i,x_i)<\varepsilon$ for every $0\leq i\leq N-1$.

Extend this orbit periodically by defining 
\[
z_{kN+i}=z_i,
\qquad
k\in\mathbb N,
\quad
0\leq i\leq N-1.
\]
Then $\{z_n\}_{n\in\mathbb N}$ is a periodic orbit of $F$. Moreover, for every $
n=kN+i$, we have
\[
d(z_n,x_n)
=
d(z_i,x_i)
<
\varepsilon.
\]
Hence, the periodic pseudo-orbit $\{x_n\}_{n\in\mathbb N}$ is $\varepsilon$-shadowed by a periodic orbit, proving that $F$ has the periodic shadowing property.
\end{proof}

\begin{remark}
\label{rem:compactness}
Notice that neither compactness of the phase space nor upper semicontinuity of the map is required in Lemma~\ref{finitePSP}. The equivalence relies solely on the periodic structure of the pseudo-orbits.
\end{remark}

We next show that periodic shadowing is preserved under inversion.

\begin{theorem}
\label{thm:periodic_inverse}
Let $(X,d)$ be a metric space and let $F:X\to 2^X$ be a set-valued map. Assume that:
\begin{enumerate}
\item $F$ is onto;
\item $F$ is continuous with respect to the Hausdorff metric;
\item $F$ has the periodic shadowing property.
\end{enumerate}
Then $F^{-1}$ also has the periodic shadowing property.
\end{theorem}

\begin{proof}
Let $\varepsilon>0$. Choose $\delta>0$ given by the periodic shadowing property of $F$.

By continuity of $F$ with respect to the Hausdorff metric, there exists $\delta_1>0$ such that $d_H(F(x),F(y))
<
\frac{\delta}{2}$ whenever $d(x,y)<\delta_1$.
Let $\{x_i\}_{i=0}^{n}$ be a finite periodic $\delta_1$-pseudo-orbit of $F^{-1}$. Then $d(x_{i+1},F^{-1}(x_i))
<
\delta_1$ for $0\leq i\leq n-1$, and $d(x_0,F^{-1}(x_n))
<
\delta_1$.

Since $F$ is onto, for each $i=0,\dots,n$ we may choose $x_i^*\in F^{-1}(x_i)$
such that $d(x_i^*,x_{i+1})
<
\delta_1$,
where $
x_{n+1}=x_0$.
By continuity, $
d_H(F(x_i^*),F(x_{i+1}))
<
\frac{\delta}{2}$.
Since $x_i\in F(x_i^*)$, it follows that $d(x_i,F(x_{i+1}))
<
\delta$ for every $0\leq i\leq n$.

Now consider the reversed sequence $w_j=x_{n-j}$ for $0\leq j\leq n$. 
The previous inequalities imply that $d(w_{j+1},F(w_j))<\delta$ for all $
0\leq j\leq n-1$, and $d(w_0,F(w_n))
<
\delta$. 

Hence, $\{w_j\}_{j=0}^{n}$ is a finite periodic $\delta$-pseudo-orbit of $F$. By Lemma~\ref{finitePSP}, there exists a periodic orbit $\{y_j\}_{j=0}^{n}$ of $F$ satisfying $
d(y_j,w_j)
<
\varepsilon$ for all $0\leq j\leq n$.

Reversing the order of this orbit, we obtain a periodic orbit of $F^{-1}$ shadowing $\{x_i\}_{i=0}^{n}$. 

Thus, $F^{-1}$ has the finite periodic shadowing property. By Lemma~\ref{finitePSP}, $F^{-1}$ has the periodic shadowing property.
\end{proof}

The periodic shadowing property is also preserved under iteration.

\begin{proposition}
Let $F:X\to 2^X$ be a set-valued map with the periodic shadowing property. Then $F^n$ has the periodic shadowing property for every
\[
n\geq1.
\]
\end{proposition}

\begin{proof}
Let $\varepsilon>0$, and let $\delta>0$ be associated to $\varepsilon$ by the periodic shadowing property of $F$.

Let $\{x_i\}_{i\in\mathbb N}$ be a periodic $\delta$-pseudo-orbit of $F^n$. By decomposing each transition of $F^n$ into $n$ consecutive transitions of $F$, we construct a periodic $\delta$-pseudo-orbit $\{y_i\}_{i\in\mathbb N}$ of $F$.

By periodic shadowing, there exists a periodic orbit $\{z_i\}_{i\in\mathbb N}$ 
of $F$ satisfying $d(z_i,y_i)
<
\varepsilon$ for all $i\in\mathbb N$. 

The subsequence $
\{z_{kn}\}_{k\in\mathbb N}$ is a periodic orbit of $F^n$ that $\varepsilon$-shadows $\{x_k\}_{k\in\mathbb N}$. 
Hence, $F^n$ has the periodic shadowing property.
\end{proof}

We conclude by proving that periodic shadowing is stable under Cartesian products.

Let $(X,F)$ and $(Y,G)$ be set-valued dynamical systems. Their product system is defined by 
\[
(F\times G)(x,y)
=
F(x)\times G(y),
\]
equipped with the metric
\[
\hat d((x_1,y_1),(x_2,y_2))
=
\max
\{
d_X(x_1,x_2),
d_Y(y_1,y_2)
\}.
\]

\begin{proposition}
\label{prop:product_psp}
Let $(X,d_X)$ and $(Y,d_Y)$ be metric spaces, and let $F:X\to 2^X$ and $G :Y\to 2^Y$ be set-valued maps. If both $F$ and $G$ have the periodic shadowing property, then $F\times G$ also has the periodic shadowing property.
\end{proposition}

\begin{proof}
Let $\varepsilon>0$. Choose $\delta=\min\{\delta_1,\delta_2\}$, where $\delta_1$ and $\delta_2$ are given by the periodic shadowing property of $F$ and $G$, respectively.

Let $\{(x_i,y_i)\}_{i\in\mathbb N}$ be a periodic $\delta$-pseudo-orbit of $F\times G$.

Then $\{x_i\}$ is a periodic $\delta$-pseudo-orbit of $F$, and $\{y_i\}$ is a periodic $\delta$-pseudo-orbit of $G$.

Hence, there exist periodic orbits $\{a_i\}$
and $\{b_i\}$ such that $d_X(x_i,a_i)<\varepsilon$ and $d_Y(y_i,b_i)<\varepsilon$ for all $i\in\mathbb N$.

Therefore, $\hat d((x_i,y_i),(a_i,b_i))
<
\varepsilon$ for all $i\in\mathbb N$. 

Since $
(a_{i+1},b_{i+1})\in (F\times G)(a_{i},b_{i})$, the sequence $\{(a_i,b_i)\}$ is a periodic orbit of $F\times G$, completing the proof.
\end{proof}

\section{Examples and Counterexamples}

The following examples illustrate the applicability of the previous
results and emphasize the necessity of the hypotheses appearing in
Theorems~\ref{theoremA} and \ref{theoremB}.

We begin with an example arising from symbolic dynamics and constructed via an isometric involution.

\begin{ex}
Let $\Sigma^+=\{0,1\}^{\mathbb{N}}$ be the one-sided full shift endowed with the metric $d(x,y)=2^{-N(x,y)}$, where
\[
N(x,y)=\min\{n\geq 0 : x_n\neq y_n\}.
\]
Define the unilateral shift map $\sigma:\Sigma^+\to\Sigma^+$ by $(\sigma(x))_n=x_{n+1}$. It is well known that $\sigma$ has both the shadowing property and by \cite{DarabiForouzanfar2017} also has periodic shadowing property. 

Define a map $R:\Sigma^+\to\Sigma^+$ by $(R(x))_n=1-x_n$ for all $n\geq 0$. 
Then $R$ is an isometric involution satisfying $R\circ\sigma=\sigma\circ R$.
 Hence, by Theorem~\ref{thm:involution}, the set-valued map $
F(x)=\{\sigma(x),R(\sigma(x))\}$ has both the shadowing property and the periodic shadowing property.
\end{ex}

The next example illustrates an application of Theorem~\ref{theoremA}
\begin{ex}
Let $X=[0,1]$ endowed with the Euclidean metric and define the set-valued map
$F:X\to 2^X$ by
\[
F(x)=
\begin{cases}
\{2x\}, & \text{if } 0\leq x<\frac12,\\[0.2cm]
\{0,1\}, & \text{if } x=\frac12,\\[0.2cm]
\{2x-1\}, & \text{if } \frac12<x\leq 1.
\end{cases}
\]
It was shown in \cite{MMT} that $F$ is positively expansive and has the shadowing property. Therefore, by Theorem~\ref{theoremA}, $F$ also has the periodic shadowing property.
\end{ex}

We now discuss the necessity of additional assumptions in Theorem~\ref{theoremB}. 
\begin{ex}\label{example4.3} Let $X=\{0,1\}$ be endowed with the discrete metric and define
\[
F(0)=\{1\},
\qquad
F(1)=\{1\}.
\]

Then $F$ has the periodic shadowing property but does not have the shadowing property. Moreover, $F$ is not chain transitive.

Indeed, let $0<\delta<1$. Since the metric is discrete, every $\delta$-pseudo-orbit is an actual orbit. The unique periodic orbit is the fixed point $(1,1,\dots)$, so every periodic $\delta$-pseudo-orbit is trivially shadowed by this orbit. Hence, $F$ has the periodic shadowing property.

However, $F$ does not have the shadowing property. Take $\varepsilon=\frac12$. For any $\delta>1$, the constant sequence $(0,0,\dots)$ is a $\delta$-pseudo-orbit since
\[
d(0,F(0))=1<\delta.
\]
If this sequence were $\varepsilon$-shadowed by an orbit of $F$, then every orbit point would have to equal $0$, which is impossible because $F(0)=\{1\}$.

Finally, $F$ is not chain transitive, since for sufficiently small $\delta>0$ there exists no $\delta$-chain from $1$ to $0$.
\end{ex}

The Example~\ref{example4.3} shows that chain transitivity in Theorem~\ref{theoremB}
cannot be omitted. Indeed, although the system satisfies the periodic
shadowing property, the lack of chain transitivity prevents the
recovery of the full shadowing property.

\end{document}